\documentclass{amsart}
                        \input xypic

\usepackage{amsthm, amsfonts, amssymb, amsmath, latexsym, enumerate, times}
\usepackage[all]{xy}
\usepackage[latin1]{inputenc}
\usepackage{mathrsfs}
\usepackage{mathtools}
\usepackage{letltxmacro}
\usepackage{comment}
\LetLtxMacro\orgvdots\vdots
\LetLtxMacro\orgddots\ddots

\makeatletter

\usepackage{array}
\usepackage{comment}
\usepackage{paralist}
\usepackage{xcolor}

\newtheorem{theorem}{Theorem}
\newtheorem{lemma}[theorem]{Lemma}

\newtheorem{corollary}[theorem]{Corollary}
\newtheorem{proposition}[theorem]{Proposition}

\theoremstyle{definition}

\newcommand{\calE}{{\mathcal E}}

\newcommand{\calH}{{\mathcal H}}

\newcommand{\calP}{{\mathcal P}}
\newcommand{\calO}{{\mathcal O}}

\newcommand{\bbF}{{\mathbb F}}
\newcommand{\bbP}{{\mathbb P}}

\newcommand{\bbZ}{{\mathbb Z}}

\newcommand{\frakB}{{\mathfrak B}}
\newcommand{\frake}{{\mathfrak e}}
\newcommand{\frakf}{{\mathfrak f}}

\def\geq{\geqslant}
\def\leq{\leqslant}

\begin{document}
 
\title[Constructions for rational multiple planes]
{Constructions for rational multiple planes}

\begin{abstract} 
A finite, normal cover $f: X\longrightarrow \bbP^2$ of degree $m\geq 3$ 
(the case $m=2$ is well known and we do not consider it in this paper) 
is called \emph{simple}, 
if there is a pencil $\mathcal P$ of rational curves of $\bbP^2$ 
such that the pull back via $f$ of $\mathcal P$ is a pencil of rational curves on $X$. 
Up to Cremona equivalence $\mathcal P$ can be assumed to be 
the pencil of lines through a fixed point $p\in \bbP^2$. 
If $\frakB$ is the branch curve of such a multiple plane, 
the general line through $p$ has to intersect $\frakB$ in $2m-2$ branch points 
(counted with multiplicities). 
If $p$ is not one of these branch points, 
then the multiple plane is said to be \ \emph{simpler}. \
In that case the branch curve will have 
a point of multiplicity $\deg(\frakB)-2m+2$ at $p$. 

In this paper we classify, under suitable generality conditions for the branch curve, 
{ simpler } triple planes 
up to Cremona equivalence
(they belong to infinitely many non--Cremona equivalent families) 
and we give examples of infinitely many non--Cremona equivalent families 
of {simpler } multiple planes of degree $m\geq 4$. 
\end{abstract}

\author{Ciro Ciliberto}
\address{Dipartimento di Matematica, Universit\`a di Roma Tor Vergata, Via O. Raimondo 00173 Roma, Italia}
\email{cilibert@axp.mat.uniroma2.it}

\author{Rick Miranda}
\address{Department of Mathematics, Colorado State University, Fort Collins (CO), 80523,USA}
\email{rick.miranda@colostate.edu}
 
\subjclass{Primary 14J26, 14M20, 14E07, 14E08, 14E20; Secondary 14E22, 14F05}
 
\keywords{multiple planes, Tschirnhausen bundle, branch curve, Cremona transformation, Hirzebruch surfaces}
 
\maketitle

\tableofcontents


\section*{Introduction}

The study of finite, normal, multiple covers of the plane 
is a classical subject in algebraic geometry. 
A problem that is still widely open is the one of classifying, 
up to Cremona transformations of the plane, 
the multiple covers $f: X\longrightarrow \bbP^2$ as above 
such that $X$ is rational. 
The first contribution to this subject goes back to the paper 
\cite {CE} of 1900 by G. Castelnuovo and F. Enriques. 
There the authors claimed to classify rational double covers of the plane 
up to Cremona transformations. 
The arguments in \cite {CE} were not complete 
and a full treatment of the subject has been given only more recently 
by L. Bayle and A. Beauville \cite {BB} in 2000 
and by A. Calabri in his PhD thesis published in 2006 (see \cite {Cal}). 
In Calabri's thesis there is also the classification, up to Cremona transformations, 
of rational cyclic triple covers of the plane. 

Besides these results, not too much is known on this subject. 
In principle, using the results in J. Blanc's thesis \cite {Bl} 
(i.e,  the classification, up to conjugation, 
of the finite abelian subgroups of the Cremona group of the plane) 
and in R. Pardini's paper \cite {Par} 
(i.e, the  description of the structure of finite, abelian covers of algebraic varieties),  
it would be possible to classify, up to Cremona transformations, 
the rational abelian covers of the plane. 
For instance in the recent preprint \cite {Cil} 
there is the classification of Galois rational covers of the plane 
with Galois group of the type $\bbZ_2^r$, with $r\geq 3$. 
However, for example, the classification, up to Cremona transformation of the plane, 
of rational non--cyclic triple planes is still unknown. 

\ In general the problem of classifying rational multiple covers of the plane 
seems to be extremely hard.  Since the case of double covers is  relatively easy,
here we deal exclusively with  covers of degree $m \geq 3$. In this paper we focus on a special family of rational multiple planes of degree $m\geq 3$. Namely we consider multiple planes $f: X\longrightarrow \bbP^2$ such that there is a pencil $\mathcal P$ of rational curves of $\bbP^2$ such that the pull--back
of the general member of $\calP$ is an irreducible rational curve on the cover.
 Let us call these multiple planes \emph{simple multiple planes}. 
 
 Up to Cremona equivalence, any such pencil of rational curves of $\bbP^2$  
is equivalent to a pencil of lines through a fixed point 
(see \cite [Prop. 5.3.5]{Cal});
hence we may assume that $\calP$ is the pencil of lines through a point $p\in \bbP^2$.

Since in principle, as we said, the cyclic case can be dealt with using \cite {Bl, Par}, 
we will focus on the non--cyclic case. \

Let $f: X\longrightarrow \bbP^2$ be a simple multiple plane of degree $m\geq 3$ as above 
in which $\calP$ is the pencil of lines through a point $p\in \bbP^2$. 
Let $\frakB\subset \bbP^2$ be the branch curve of the multiple plane, 
which has even degree. 
Since the pull--back to $X$ of a general line $r$ through $p$ is irreducible of genus $0$, 
the number of branch points on $r$, given by Riemann--Hurwitz formula, must be $2m-2$, 
if we count each of them with the appropriate multiplicity. 
If $p$ does not count as one of these branch points 
and if the branch curve is reduced, 
we will say that the simple multiple plane is \ \emph{simpler}. \
In that case the branch curve $\frakB$ has to meet the general curve through $p$ 
in $2m-2$ points off $p$, 
so it must have a point of multiplicity $\deg(\frakB)-2m+2$ at $p$. 
A first step in the classification would be to 
classify (non--cyclic)  \ simpler \ multiple planes of degree $m\geq 3$. 

Note that the consideration of {simpler } multiple planes 
$f: X\longrightarrow \bbP^2$ of degree $m$ 
is equivalent to  the consideration of  multiple covers 
$\pi: S\longrightarrow \bbF_1$ of degree $m$ 
such that the branch curve $B$ is reduced 
and does not contain the $(-1)$--curve $E$ on $\bbF_1$ 
and intersects the general fibre $F$ of the structure morphism 
$\bbF_1\longrightarrow \bbP^1$ in $2m-2$ points. 
We will make the generality hypothesis that $B$ on $\bbF_1$ 
has only simple cusps as singularities and we will often take this viewpoint. 

In this paper we first focus on the case $m=3$ 
and we do classify {simpler } triple planes, 
at least under some generality assumptions 
(partly specified above, partly in the course of the paper). 
First, in Section \ref {sec:trip}, 
using the general theory of triple covers as exposed in \cite{Miranda85}, 
we construct infinitely many families of such rational triple planes. 
Then in Section \ref {sec:geom} we give a geometric interpretation of this construction, 
and, using it, in Section \ref {sec:conv} 
we prove that the {simpler } triple planes we constructed 
are the only such triple planes, thus proving a classification theorem. 
In Section \ref {sec:mult}, 
inspired by the geometric description of the above triple planes, 
we construct infinitely many families of {simpler } multiple planes 
of any degree $m\geq 4$. 
Finally, in Section \ref {sec:ce} we prove that 
sufficiently general multiple planes as above 
with branch curves of different degree are not Cremona equivalent. 

We stress that, up to our knowledge, 
so far even the existence of infinitely many families 
of non--Cremona equivalent rational multiple planes of a given degree $m\geq 3$
was not known. 

{\
The reader may ask about the more restrictive case
where a multiple plane $f:X \to \bbP^2$
has a rational pre-image of the general line,
not just the general curves in a pencil of rational curves.
The analysis of the general situation here is rather straightforward:
since $X$ has a net of rational curves,
the complete linear system of that net
maps $X$ to a rational scroll or a Veronese surface,
and the map $f$ is obtained as a projection of such a surface.
If the projection is general, the numerical characters
(degree, singularities) are classically computed;
it is however an open question about special positions for the projection,
which could provide a list of allowable degenerations of the branch locus
for such a projection.  
These are not classified,
except for triple covers with degree $4$ branch curve, 
where the branch curve is either an irreducible quartic with three ordinary cusps,
or the union of a cuspidal cubic and its (unique) flexed line; see \cite{CM26}.
}
\medskip

\noindent {\bf Acknowledgements:} C. Ciliberto is a member of GNSAGA of the Istituto Nazionale di Alta Matematica. \medskip

\section{Constructing {simpler } triple covers of the plane}\label{sec:trip}

\subsection{The algebraic construction} In this section we will freely use the results in \cite {Miranda85}.

Fix strictly positive integers $x,y$ with $2x > y$ and $2y > x$,
and set $e=x+y$.
Let $Y=\bbF_1$ be the blowup of $\bbP^2$ at one point, which is a ruled surface;
call the exceptional divisor $E$ and the fiber class $F$.
The pull--back to $Y$ of a general line of the plane has class $H = E+F$.
Consider the decomposable rank two bundle on $Y$
\begin{equation}\label{eq:split}
\calE = \calO_Y(-E-xF)\oplus\calO_Y(-E-yF).
\end{equation}
The Chern classes of $\calE$ are computed to be
\begin{equation}
c_1(\calE) = -2E-(x+y)F=-2E-eF \;\;\;\text{ and }\;\;\; c_2(\calE) = x+y-1 = e-1
\end{equation}
and the standard intersection numbers are
\begin{equation}
c_1(\calE)^2 = 4e-4;\;\;\; c_1(\calE)\cdot K_Y = (-2E-eF)\cdot(-2E-3F) = 2e+2.
\end{equation}
We note that
\[
S^3\calE^* \cong 
\calO_Y(3E+3xF) \oplus 
\calO_Y(3E+(2x+y)F) \oplus 
\calO_Y(3E+(x+2y)F) \oplus 
\calO_Y(3E+3yF)
\]
and
\[
\wedge^2\calE \cong \calO_Y(-2E-(x+y)F) = \calO_Y(-2E-eF)
\]
so that
\[
S^3\calE^* \otimes \wedge^2\calE \cong
\calO_Y(E+(2x-y)F) \oplus 
\calO_Y(E+xF) \oplus 
\calO_Y(E+yF) \oplus 
\calO_Y(E+(2y-x)F).
\]

A triple cover of $Y$ can be built from a general section of this rank $4$ bundle;
our assumptions on $x$ and $y$ imply that each of the four line bundles have smooth general movable members, and in fact are very ample if $x,y \geq 2$.

A section of this rank four bundle consists of four sections
\begin{gather*}
a \in H^0(\calO_Y(E+xF)); \\
b \in H^0(\calO_Y(E+(2x-y)F)); \\
c \in H^0(\calO_Y(E+(2y-x)F); \\
d \in H^0(\calO_Y(E+yF)) 
\end{gather*}
which give the structure constants of the $\calO_Y$-algebra 
defining the cover $\pi:S\to Y$.
If we set
\begin{align*}
A &= a^2-bd \in H^0(2E+2xF); \\
B &= ad-bc \in H^0(2E+eF)); \\
C &= d^2-ac \in H^0(2E+2yF); \\
D &= B^2-4AC \in H^0(4E+2eF),
\end{align*}
the branch locus for the cover $\pi$ is the divisor $B$ defined by $D=0$;
its class is $-2c_1(\calE) = 4E+2eF$. If the above sections are general enough, $B$ is  irreducible and reduced and has only simple cusps as singularities. The number of these cusps is $3c_2(\calE) = 3e-3$.
In the rest of the paper we will suppose that these generality assumptions are verified.

In conclusion we have proved:

\begin{proposition}\label{prop:exist} 
Given two positive integers $x,y\geq 2$,
with $2x>y$ and $2y>x$,
there are triple covers $\pi: S\longrightarrow Y=\bbF_1$, 
with Tschirnhausen bundle given by \eqref {eq:split} and branch curve $B\in 4E+2eF$ (where $e=x+y$) which is irreducible and reduced.
\end{proposition}

\subsection{The pre-images of the fibers}\label{sec:fibres}
The general curves in the linear system $|F|$ of $Y=\bbF_1$ 
lift on $S$ to curves $C$ which are triple covers of $\bbP^1$ 
branched at $B\cdot F = (4E+2eF)\cdot F = 4$ points.
These curves are therefore rational, by the Riemann--Hurwitz formula.
We conclude therefore that $S$ has a rational pencil of rational curves $|C|$ 
and therefore it is rational by Noether's theorem. 

\subsection{Invariants of the triple cover}

Using Riemann-Roch, we see that the arithmetic genus of $B$ is
\[ 
p_a(B) = \frac{1}{2}(B^2+B\cdot K_Y) + 1 =6e-9. \\
\] 

The genus of the ramification divisor $R$ on $X$ (which is smooth) is
\[ 
g(R) = 2c_1(\calE)^2 - c_1(\calE)\cdot K_Y + 1 - 3c_2(\calE) = 3e-6\\
\] 
and this is also the geometric genus $g(B)$ of $B$. 
This is consistent with the fact that $g(B)=p_a(B)-(3e-3)$, 
where $3e-3$ is the number of cusps and these are the only singularities of $B$. 

Using Proposition 10.3 of \cite{Miranda85}, we have the following.
\begin{align*}
\chi(\calO_S) &= 3\chi(\calO_Y) + \frac{1}{2}c_1(\calE)^2 - \frac{1}{2}c_1(\calE)\cdot K_Y - c_2(\calE) = 1,\\
K_S^2 &= 3K_Y^2 -4c_1(\calE)\cdot K_Y + 2c_1(\calE)^2 - 3c_2(\calE) = 11-3e, \\
e(S) &= 3e(Y) - 2c_1(\calE)\cdot K_Y + 4c_1(\calE)^2 - 9c_2(\calE) -3e+1. \\
\end{align*}

The fact that $\chi(\calO_S) = 1$ agrees with the fact that $S$ is rational.
The other invariants are consistent with $S$ being a rational surface 
which is a blowup of $\bbP^2$ at $3e-2$ (distinct or infinitely near) points 
or the blow--up of a Hirzebruch surface $\bbF_a$ 
at $3e-3$ (distinct or infinitely near)  points.

\subsection{The pre-image of $E$}\label{sec:prb}
The branch curve $B$ meets the $(-1)$-curve $E$ in
$(4E+2eF)\cdot E = 2e-4$ points. 
We will make the further assumption that these are distinct. 
Then the pre-image $G$ of $E$
is a smooth triple cover of $E$ branched at $2e-4$ points,
and hence has genus satisfying the Riemann--Hurwitz formula
\[
2g(G)-2 = 3(-2) + (2e-4) = 2e-10
\]
implying that $G$ is a curve of geometric genus $g(G) = e-4$.
Since $\pi^*(E) = G$, we have  $G^2 = -3$.

As long as $e \geq 4$, 
we can make a contraction $\phi: S\longrightarrow X$ of the curve $G$ 
(and simultaneously contract $E$ at a point $p$ of $\bbP^2$)
to obtain a map $f:X  \longrightarrow \bbP^2$
which is the desired {simpler } triple plane,
where $X$ is singular (at the image of the curve $G$). 
The branch curve $\frakB$ of $f:X  \to \bbP^2$ has degree $2e$, 
has an ordinary point of multiplicity $2e-4$ at $p$ 
and in addition $3e-3$ {\ ordinary } cusps. 


\subsection{The pre-image of the line class $H$}\label{sec:prh}
The branch curve $B$ meets the line class $H= E+F$ in $(4E+2eF)\cdot H = 2e$ points.
The  pre-image of a general curve in $|H|$ is therefore a smooth curve $M$ 
which has  genus $g(M)$ satisfying
$2g(M)-2 = 3(-2)+2e$ and hence $g(M) = e-2$;
since $H^2=1$, we have $M^2 = 3$.
Hence $M\cdot K_S = 2e-9$.
By Riemann-Roch, we see that
\begin{equation}\label{eq:longa}
\chi(\calO_S(M)) = 1+(M^2-M\cdot K_X)/2 = 1+(3-2e+9)/2 = 7-e.
\end{equation}

\begin{proposition}\label{prop:ex} 
One has $h^0(\calO_S(M))=3$ and therefore $h^1(\calO_S(M))=e-4$.
\end{proposition}

\begin{proof} 
Since $H\cdot E=0$, we have $M\cdot G=0$. 
Moreover since $H=E+F$, we have $M=G+C$. 
Consider the exact sequence
\begin{equation}\label{eq:long}
0\longrightarrow \mathcal O_S(C) \longrightarrow \mathcal O_S(M) \longrightarrow \mathcal O_G(M)\cong \mathcal O_G\longrightarrow 0.
\end{equation}
We have $h^0(\mathcal O_S(C))=2$, and $h^i(\mathcal O_S(C))=0$, for $1\leq i\leq 2$.
From the long exact sequence of \eqref {eq:long}, 
we immediately deduce that $h^0(\calO_S(M))=3$ and 
$$h^1(\calO_S(M))=h^1(\calO_G)=g(G)=e-4,$$ 
as desired. This also follows from $h^0(\calO_S(M))=3$ and \eqref {eq:longa}.
\end{proof}

The map to $\bbP^2$ is of course given by the system $|M|$;
in the next section we will give a geometric construction for this system.

\section{Geometric description}\label{sec:geom}

In this section we are keeping all the generality assumptions 
we made so far on the branch curve $B$. 

\subsection{Contraction to a Hirzebruch surface} 
As we saw in Section \ref {sec:fibres}, 
there is a rational pencil $|C|$ of rational curves on $S$, 
i.e., the pull backs of curves in $|F|$ on $Y=\bbF_1$.
We can blow down the ruled surface $S$
to a minimal rational ruled surface $\bbF_h$ for a suitable $h\geq 0$,
obtaining a morphism $g:S \to \bbF_h$,
by contracting (in sequence) $(-1)$--curves which are components of curves in $|C|$.
The number of contractions is $3e-3$ because, as we saw, $K^2_S=11-3e$. 

\begin{lemma}\label{lem:contr} 
It is possible to make the contraction $g: S\longrightarrow \bbF_h$ 
in such a way that the image $\bar G$ of $G$ on $\bbF_h$ is still smooth. 
\end{lemma}

\begin{proof} Only the reducible members of $|C|$ are affected by the contractions.
We must show that we can achieve the reduction to the minimal surface $\bbF_h$
by only contracting $(-1)$--curves that meet $G$ with intersection number at most one. 

For a curve $C_0\in |C|$, let $F_0=\pi(C_0)$. 
We have a morphism $\pi_{|C_0}: C_0\longrightarrow F_0\cong \bbP^1$, that is finite. Hence, if $C_0$ is reducible, it consists of at most three irreducible components, 
all smooth and rational. 

If $C_0$ consists of two irreducible components, 
one of them intersects $G$ in one point, the other in two points, 
and both are $(-1)$--curves. 
In the contraction $g: S\longrightarrow \bbF_h$ 
we can choose to contract only the $(-1)$--curve that intersect $G$ in one point. 

If $C_0$ consists of three irreducible components, 
each of these intersects $G$ in one point. 
Let us call $C_{01},C_{02},C_{03}$ the three components of $C_0$, 
with the property that $C_{01}\cdot C_{02}= C_{02}\cdot C_{03}=1$, $C_{01}\cdot C_{03}=0$.
There are two cases:\\
\begin{inparaenum}[(i)]
\item $C_{01}^2=C_{03}^2=-1, C_{02}^2=-2$;\\
\item $C_{01}^2=C_{03}^2=-2, C_{02}^2=-1$.\\
\end{inparaenum}
In case (i), we choose to contract $C_{01}$ and $C_{03}$. In case (ii) we choose to contract first $C_{02}$ and then the image of one of the other two curves. 

With these choices, the image $\bar G$ of $G$ in $\bbF_h$ is easily seen to be smooth.
\end{proof}

Note that
\begin{equation}\label{eq:g2}
\bar G^2=G^2+(3e-3)=3e-6.
\end{equation}
On this curve there is a (curvilinear) $0$--dimensional subscheme $Z$ of length $3e-3$ 
that is the contraction of the $3e-3$ curves $(-1)$-curves.  

Let us denote by $\frakf$ the class of the ruling of $\bbF_h$ 
and by $\frake$ the class of the section with $\frake^2=-h$. 
The classes $\frake$ and $\frakf$ generate the Picard group of $\bbF_h$.
Since $G\cdot F=3$, we have $\bar G\cdot \frakf=3$, 
hence we have a relation of the form
\begin{equation}\label{eq:g22}
\bar G\sim 3\frake+ \ell \frakf,
\end{equation}
for a suitable positive integer $\ell$. 
From \eqref {eq:g2} and \eqref {eq:g22} we deduce that
\begin{equation}\label{eq:e}
e=2\ell-3h+2.
\end{equation}
Moreover since $0\leq \bar G\cdot \frake=\ell-3h$ we have $\ell\geq 3h$. 
The solutions to the equation \eqref {eq:e} in $\ell$ and $h$ are of the form
$$
h=e+2k, \,\, \ell=2e-1+3k, \,\, \text{with $k$ an integer}.
$$
However, since we have $\ell\geq 3h$, we must have $3k+e\leq -1$. 
Moreover $h\geq 0$ yields $2k\geq -e$, so that
$$
-\frac{e}{2} \leq k\leq -\frac {e+1}{3}.
$$

\subsection{The linear system $\calH$ on $\bbF_h$}\label{ssec:H}  
We denote by $\calH$ the image via $g: S\longrightarrow \bbF_h$ 
of the linear system $|M|$ on $S$, 
where we recall that we set $M=\pi^*(H)$, 
with $H$ the pull back to $Y=\bbF_1$ of a general line of $\bbP^2$. 
Note that 
$$
\dim (\calH)=\dim (|M|)=2.
$$
We will denote by $\bar H$ the general curve in $\calH$.

Note that, since  $H\sim F+E$ on $Y$, we have $M\sim G+C$ on $S$. 
So all $(-1)$--curves we contracted in the map $g: S\longrightarrow \bbF_h$, 
that intersect $G$ in one point each, 
also intersect $M$ in one point each. 
This implies that the general curve $\bar H$ of $\calH$ is also smooth, 
and it cuts out on $\bar G$ a divisor containing the length $3e-3$ scheme $Z$. 

Moreover
$$
\bar H^2=M^2+3e-3=3e.
$$
Note next that 
$$
\bar H\sim G+\frakf \sim 3\frake+ (\ell+1) \frakf,
$$ 
and therefore
$$
\bar H\cdot \bar G=(3\frake+ (\ell+1) \frakf)\cdot (3\frake+ \ell \frakf)=3e-3
$$
(see \eqref {eq:e}). 
So $Z$ is the complete intersection of a general curve $\bar H$ of $\calH$ and $\bar G$. 

\begin{proposition}\label{prop:geom} 
The linear system $\calH$ defines a rational map 
$$
\varphi_{\calH}: \bbF_h\dasharrow \bbP^2
$$
such that there is  a commutative diagram 
\begin{equation}\label{eq:diag}
  \xymatrix{
    S \ar[d]_{g}  \ar[rr]^{\pi}  &  &  \ar[d]^{p} Y  \\
  \bbF_h  \ar@{-->}[rr]^{\varphi_{\calH}} & &\bbP^2
}
\end{equation}
where $p: Y=\bbF_1\longrightarrow \bbP^2$ is the obvious contraction of $E$. 
\end{proposition}

\begin{proof} 
The proof is immediate: note that $g: S\longrightarrow \bbF_h$ 
is just the blow up of $\bbF_h$ along the cluster determined by $Z$. 
\end{proof}

Proposition \ref {prop:geom} gives a complete geometric description 
of the triple planes we constructed in Section \ref {sec:trip}. 
Namely, one fixes an integer $h \geq 0$,
and takes a smooth curve $\bar G$ in the linear system $3\frake+\ell\frakf$,
where $\ell \geq 3h$.
We then take a general member $\bar H$ in the linear system
$|\frake+(\ell+1)\frakf|$,
and consider the complete intersection scheme $Z = \bar G \cap \bar H$.
We then consider the linear system $\calH$ consisting of those members of
the complete linear system $|\frake+(\ell+1)\frakf|$
that contain the scheme $Z$.
That is a two-dimensional linear system and gives the triple cover map to the plane.

\section{The converse analysis}\label{sec:conv}

In this section we consider a {simpler } triple plane 
$f: X\longrightarrow \bbP^2$, 
and the corresponding triple cover $\pi: S\longrightarrow Y=\bbF_1$,  
such that the branch curve $B$ is such that $B\cdot F=4$. 
We will assume that $B$ is \emph{sufficiently general}, 
i.e., irreducible and reduced with at most ordinary cusps as singularities, 
intersecting $E$ transversally at distinct points. 
In this case we will say that the {simpler } triple plane is also \emph{general}. 
In that case $S$ will be smooth.  We have 
\[
B\sim 4E+2eF,
\]
for some positive integer $e\geq 2$ and $B\cdot E=2e-4$. 

The main purpose of this section is to prove the following:

\begin{theorem}\label{thm:conv} 
The construction in Section \ref {sec:trip} gives all {simpler } general triple planes 
$f: X\longrightarrow \bbP^2$. 
\end{theorem} 

\begin{proof} There is a commutative diagram
\begin{equation}\label{eq:comm}
  \xymatrix{
    S \ar[d]_{\phi}  \ar[rr]^{\pi}  &  &  \ar[d]^{p} Y  \\
X  \ar[rr]^{f} & &\bbP^2
}
\end{equation}
where $\phi: S\longrightarrow X$ is the contraction of the curve $G=\pi^*(E)$. 

Arguing as in \S \ref {sec:fibres} we see that the pull back to $S$ 
of the pencil $|F|$ of $Y$ is a linear system of rational curves $|C|$. 
Hence $S$ is rational by Noether's theorem.

As in \S \ref {sec:prb}, 
we see that (under the generality hypotheses on $B$) 
the curve $G$ on $S$ is a smooth curve of genus $e-4$ 
and self--intersection $G^2=-3$.

As in \S \ref {sec:prh} we see that the pre-image of a general curve in $|H|$ 
is a smooth curve $M$ with genus $e-2$. 
Moreover $M\sim G+C$ and $\dim(|M|)=2$. 

Let $\calE$ be the Tschirnhausen bundle of $\pi: S\longrightarrow Y$. 
We have
\begin {equation}\label{eq:c1}
c_1(\calE) = -2B-eF.
\end{equation}
By the formula
\[
\chi(\calO_S) = 3\chi(\calO_S) + \frac{1}{2}c_1(\calE)^2 - \frac{1}{2}c_1(\calE)\cdot K_S - c_2(\calE),
\] 
since $\chi(\calO_S)=\chi(\calO_S)=1$ 
and by \eqref {eq:c1}, we deduce that $c_2(\calE)=e-1$, 
so that the branch curve $B$ has $3c_2(\calE)=3e-3$ cusps. 
Then the computation $K_S^2=11-3e$ is still valid. 

At this point we can proceed just as in Section \ref {sec:geom}. 
Namely, we can blow down exactly $3e-3$ exceptional
$(-1)$--curves  in fibres of $|C|$ getting a morphism $g: S\longrightarrow \bbF_h$, 
with a suitable $h\geq 0$, 
and we can do this in such a way that the image $\bar G$ of $G$ via $g$ 
is still a smooth curve with genus $e-4$ 
whose self intersection is given by \eqref {eq:g2}. 
On $\bar G$ there is a $0$--dimensional subscheme $Z$ of length $3e-3$ 
that is the contraction of the $3e-3$ exceptional $(-1)$--curves. 

With the same notation as in Section \ref {sec:geom}, 
we still have \eqref {eq:g22}  and \eqref {eq:e}.  
The image $\mathcal H$ linear system of $|M|$ on $\bbF_h$ 
is  a $2$--dimensional  sublinear system of 
$|3\mathfrak e+ (\ell+1) \mathfrak f|$  
with simple base points at the subscheme $Z$ of $\bar G$ 
that are the complete intersection of $\bar G$ 
with a general curve $\bar H$ in $\mathcal H$. 
The commutative diagram \eqref {eq:diag} still holds. 

In conclusion the triple cover $\pi: S\longrightarrow Y=\bbF_1$ 
is exactly \ the same \  that we constructed  in Section \ref {sec:trip}, 
and this implies that the Tschirnhausen bundle 
necessarily splits as in \eqref {eq:split}, with $x+y=e$. 
\end{proof}

\section{Multiple planes}\label{sec:mult}

Consider in general multiple covers $\pi: S\longrightarrow Y=\bbF_1$ 
of degree $m\geq 4$ (the case $m=2$ is trivial) 
such that the branch curve $B$ is irreducible reduced 
and such that $B\cdot F=2m-2$.  
By imitating what happens for the case $m=3$, 
in this section we are going to show how to construct infinitely many families of such covers. Note that 
\[
B\sim (2m-2)E+2eF,
\]
for some positive integer $e\geq m-1$, so that $B\cdot E=2e-2m+2$.  

There is a commutative diagram \eqref {eq:comm}
where $f:  X\longrightarrow \bbP^2$ 
is a {simpler } multiple plane of degree $m$ 
for which there is a branch curve $\mathfrak B$ of degree $2e$ 
with a point of multiplicity $2e-2m+2$. 
The consideration of such a multiple plane 
is equivalent to the consideration 
of the multiple cover $\pi: S\longrightarrow Y$ of degree $m$.

\begin{theorem}\label{thm:ext} 
For every $m\geq 2$ there exist {simpler } multiple planes of degree $m$ as above.
\end{theorem}

\begin{proof}
Consider a Hirzebruch surface $\bbF_h$ 
and take a smooth curve $\bar G$ on $\bbF_h$ such that 
$$
\bar G\sim m\frake+\ell \frakf, \,\, \text{with $\ell$ a positive integer}.
$$
For this it is necessary that $0\leq \bar G\cdot \frake = \ell -mh$. 
Then consider the linear system $|m\frake+(\ell+1) \frakf|$ 
(which is very ample) and take a general curve $\bar H$ in this system, 
that intersects $\bar G$ in
$$
\bar H\cdot \bar G=
(m\frake+\ell \frakf)\cdot (m\frake+(\ell+1) \frakf)
=m(2\ell+1-mh)
$$
points (that in general will be distinct) 
forming a $0$--dimensional subscheme $Z$ of $\bar G$. 
Consider then the linear subsystem $\calH$ of $|\bar H|$, 
formed by all curves that contain $Z$. 

One has $\dim (\calH)=2$ 
because $\bar G$ splits off $\calH$ with one condition 
and the residual system is the pencil $|\frakf|$.  
The map $\varphi_{\calH}: \bbF_h\dasharrow \bbP^2$ 
gives a cover of $\bbP^2$ of degree $m$. 
More precisely,  there is a cartesian diagram \eqref {eq:diag}
where $f: S\longrightarrow Y$ is a degree $m$ cover, 
and $S$ is the blow--up of $\bbF_h$ along $Z$. 
The cover $\pi: S\longrightarrow Y=\bbF_1$ is a desired one. 
\end{proof}

It is conceivable that the construction in this section 
gives all the degree $m$ covers $f: S\longrightarrow Y$
whose branch curve $B$ is reduced and such that $B\cdot F=2m-2$, 
as it happens for $m=2$ (trivial) and for $m=3$ (see Section \ref {sec:conv}), 
or at least those for which $B$ is sufficiently general.

\section{Cremona equivalence}\label{sec:ce}

Given two multiple planes 
$\pi: X\longrightarrow \bbP^2$ and $\pi': X'\longrightarrow \bbP^2$ 
we will say that they are \emph{Cremona equivalent} 
if there is a commutative diagram of the form
\begin{equation}\label{eq:loh}
\xymatrix{
X' \ar[d]_{\pi'}\ar@{-->}[r]^\phi &
X\ar[d]^\pi\\
\bbP^2 \ar@{-->}[r]^\varphi& \bbP^2   }
\end{equation}
such that $\phi:X'\dasharrow X$ and $\varphi: \bbP^2\dasharrow \bbP^2$ 
are birational maps. 

\begin{theorem}\label{thm:noncr} 
Consider two multiple planes of degree $m\geq 3$ of the form 
$\pi: X\longrightarrow \bbP^2$ and $\pi': X'\longrightarrow \bbP^2$ 
of the type we constructed in Section \ref {sec:mult},
with branch curves $\frakB$ and $\frakB'$ respectively, 
of degrees $2e$ and $2e'$ respectively, 
with an ordinary point $p$ of multiplicities $2e-2m+2$ and $2e'-2m+2$ respectively. 
These two multiple planes are Cremona equivalent 
if and only if they are isomorphic, 
i.e., if and only if $\varphi$ is a projective transformation 
and the diagram \eqref {eq:loh} is cartesian. 
\end{theorem} 

\begin{proof}

We may assume that $e'\geq e$.
Suppose there is a diagram like \eqref {eq:loh}. 
Then $\varphi: \bbP^2\dasharrow \bbP^2$ has to preserve 
the pencil of lines with centre $p$, 
so it must be a De Jonqui\'eres transformation, 
defined by a linear system of curves of a certain degree $d\geq 1$ 
with a point of multiplicity $d-1$ at $p$ and further $2d-2$ simple base points 
(that can be proper or infinitely near). 
We want to prove that $d=1$. 

The De Jonqui\'eres transformation contracts to $p$ 
the unique rational curve $\Delta$ of degree $d-1$, 
with a point of multiplicity $d-2$ at $p$ 
and passing through the $2d-2$ simple base points. 
Then $\Delta$ has to intersect $\frakB'$ 
in $2e-2m+2$ points off the aforementioned base points. 
Hence we have
$$
2e-2m+2=2e'(d-1)-(2e'-2m+2)(d-2)-x,\,\, \text {with}\,\, x\leq 2d-2;
$$
therefore we deduce that
$$
e\geq e'+(m-2)(d-1)\geq e'+d-1.
$$
Since $e'\geq e$, we find $d\leq 1$, as wanted. 
\end{proof}

{\
\begin{corollary}
There exist infinitely many families of non-Cremona equivalent rational multiple planes of any given degree $m \geq 3$.
\end{corollary}
} 

\end{document}